\documentclass[12pt]{article}
\usepackage{amsmath,amsfonts,amssymb,amsthm,bbm}
\usepackage{booktabs,epsf,epsfig,psfrag,color}
\usepackage[ansinew]{inputenc}
\usepackage[active]{srcltx}

\DeclareMathOperator{\dist}{dist}
\DeclareMathOperator{\graph}{graph}
\DeclareMathOperator{\Proj}{proj}

\begin{document}

\newtheorem{oberklasse}{OberKlasse}
\newtheorem{lemma}[oberklasse]{Lemma}
\newtheorem{proposition}[oberklasse]{Proposition}
\newtheorem{corollary}[oberklasse]{Corollary}
\newtheorem{definition}[oberklasse]{Definition}
\newtheorem{example}[oberklasse]{Example}
\newtheorem{hypothesis}{Hypothesis}
\renewcommand*{\thehypothesis}{\Alph{hypothesis}}
\newcommand{\R}{\ensuremath{\mathbbm{R}}}
\newcommand{\N}{\ensuremath{\mathbbm{N}}}
\newcommand{\Z}{\ensuremath{\mathbbm{Z}}}
\renewcommand{\epsilon}{\varepsilon}

\title{The Euler scheme for state constrained ordinary differential inclusions}
\author{Janosch Rieger\footnote{Department of Mathematics, Imperial College London, 180 Queen's Gate, London SW7 2AZ, United Kingdom}}
\date{\today}
\maketitle

\begin{abstract}
We propose and analyze a variation of the Euler scheme for state constrained ordinary differential inclusions
under weak assumptions on the right-hand side and the state constraints.
Convergence results are given for the space-continuous and the space-discrete versions of this scheme, and 
a numerical example illustrates in which sense these limits have to be interpreted.
\end{abstract}



\section{Introduction}

Numerical methods for ordinary differential inclusions have mainly been considered in the unconstrained case
and for convex-valued right-hand sides.
Summaries of early results can be found in the surveys \cite{Dontchev:Lempio:92} and \cite{Lempio:Veliov:98}.
An approach using subdivision and continuation techniques for time-independent systems with affine controls 
was given in \cite{Szolnoki:03}.

The variation of the Euler scheme proposed in \cite{Grammel:03} only uses extremal points of the right-hand side,
so that a fully discretized scheme is obtained for polytope-valued right-hand sides.
The high complexity of this method motivated the detailed analysis of spatial discretization effects on 
the Euler scheme given in \cite{Beyn:Rieger:07}.
A further reduction of computational costs can be achieved by tracking the boundaries of the reachable sets 
of the fully discretized Euler scheme using only lower-dimensional data of the right-hand side, see \cite{Rieger:14}.

A very different type of algorithm has been proposed in \cite{Baier:Gerdts:Xausa:13}, where a reachable set is 
computed by an optimal control routine that minimizes the distance to grid points in the relevant region at the
terminal time. 
As the optimizer may find a local instead of a global minimum, the propriety of this method cannot be guaranteed.

An implicit Euler scheme has been analyzed in \cite{Beyn:Rieger:10}, which is considerably more efficient than
the explicit Euler scheme when applied to stiff differential inclusions. 
The disadvantage of this method is that it has to solve an algebraic inclusion in every time step.
This problem is successfully avoided by the semi-implicit Euler schemes discussed in \cite{Rieger:14a},
leading to a better performance of the algorithm.

The Euler scheme for differential inclusions with nonconvex right-hand sides has been investigated in \cite{Sandberg:08}.
Error estimates for the Euler scheme applied to a differential inclusion with convex, but one-sided Lipschitz 
right-hand side have been published in \cite{Donchev:Farkhi:98}.

\medskip

Up to our knowledge, the only convergence analysis of the Euler scheme for ordinary differential inclusions with state 
constraints has been given in the paper \cite{Baier:Chahma:Lempio:07}.
The analysis is based on stability theorems that quantify the relationship between the solution set of the 
constrained problem and the same differential inclusion without constraints. 
These results play an important role in optimal control theory and are still a subject of intensive research, 
see \cite{Bettiol:Bressan:Vinter:10}, \cite{Bettiol:Frankowska:Vinter:12} and the references therein.

The drawback of these stability theorems in the present context are the strict assumptions on the compatibility
between right-hand side and constraints that are needed to obtain such a result.
In applications like collision avoidance, see \cite{Gerdts:Xausa:13}, 
\cite{Landry:Gerdts:Henrion:Hoemberg:13} and the references therein, such assumptions do certainly not hold.

\medskip

The present article is motivated by the lack of a convergence analysis for the Euler scheme in the present of state 
constraints under weaker conditions.
In Section \ref{setting:notation} we fix the notation, give a formal problem statement and formulate our basic
assumptions, that the initial set is compact, the right-hand side is Lipschitz, the state constraints are upper 
semicontinuous in time and that all data are compatible in the very weak sense that there is a nontrivial time 
interval on which a solution of the state constrained inclusion exists.

Under these conditions, we prove some auxiliary results in Section \ref{auxiliary}, which are used in
Section \ref{inflated} to show that the solution sets of the inflated Euler scheme converge towards the 
exact solution sets in Hausdorff distance.
In Section \ref{space:disc}, we show that the same holds for a spatially discretized version
of the scheme provided all constraints are relaxed in a proper way.
A numerical example given in Section \ref{num:ex} illustrates in which sense this limit has to be interpreted.
Note that convergence results for the inflated Euler scheme without spatial discretization are of interest for 
first-discretize-then-optimize strategies in optimal control and algorithms such as the one proposed in \cite{Baier:Gerdts:Xausa:13}, 
which consider the Euler scheme as a constraint of a finite-dimensional optimization problem.

At the end of the paper, in Section \ref{stability:sec}, we give a brief explanation how stability results in
the spirit of \cite{Bettiol:Bressan:Vinter:10} and \cite{Bettiol:Frankowska:Vinter:12} automatically yield 
linear convergence of the inflated Euler scheme in Hausdorff distance.

\section{Setting and notation} \label{setting:notation}

Throughout this paper, the Euclidean norm and the modulus will be denoted by $\|\cdot\|:\R^d\to\R_+$ and $|\cdot|:\R\to\R_+$.
As some of the following concepts will be used not only in the state space $\R^d$, but also in the space of continuous functions, 
we introduce them in a general Banach space $E$.

Let $A,B\subset E$ be compact. Then
\[\dist(A,B):=\max_{a\in A}\min_{b\in B}\|a-b\|\]
is called the Hausdorff semidistance between $A$ and $B$.
In the particular case, where $A=\{a\}$ with $a\in E$, the Hausdorff semidistance is the usual distance $\dist(a,A)$
from a point to a set, and we denote 
\[B_R(A):=\{x\in E:\dist(x,A)\le R\}.\]
The Hausdorff distance
\[\dist_H(A,B):=\max\{\dist(A,B),\dist(B,A)\}\]
is a metric on the compact subsets of $E$.
The metric projection from $x\in E$ to $A$ is the set
\[\Proj(x,A):=\{a\in A: \|x-a\|=\dist(x,A)\}.\]
Since $A$ is compact, the metric projection is nonempty. 
If, in addition, $E$ is a Hilbert space and $A$ is convex, then $\Proj(x,A)$ is a singleton.

By $W^{1,1}([0,\tau],\R^d)$ with $\tau\in(0,\infty)$, we denote the Banach space of absolutely continuous functions from $[0,\tau$] to $\R^d$, 
i.e.\ the space of all functions, which possess a weak derivative in $L^1((0,\tau),\R^d)$.

\medskip

Let $F:[0,T]\times\R^d\rightrightarrows\R^d$ be a set-valued mapping. 
A solution to the ordinary differential inclusion
\begin{align} 
\dot x(t) &\in F(t,x(t)) \label{ODI}
\end{align}
is an absolutely continuous function $x(\cdot)\in W^{1,1}([0,T],\R^d)$ satisfying \eqref{ODI} almost everywhere in $[0,T]$.
Depending on the application, solutions may be required to satisfy an initial condition
\begin{align}
x(0) &\in X_0 \label{ODIIC}
\end{align}
with initial set $X_0\subset\R^d$ and state constraints
\begin{align} 
x(t) &\in A(t) \label{ODISC}
\end{align}
with $A:[0,T]\rightrightarrows\R^d$.
For $\tau\in(0,T]$, we denote the solution sets of the unconstrained and the constrained initial value problems by
\begin{align*}
S^u(\tau) := &\{x(\cdot)\in W^{1,1}([0,\tau],\R^d): \text{\eqref{ODI} holds a.e.\ in $(0,\tau)$,\ \eqref{ODIIC} holds}\},\\
S^c(\tau) := &\{x(\cdot)\in S^u(\tau): x(t)\ \text{satisfies \eqref{ODISC} for all}\ t\in[0,\tau]\}.
\end{align*}
These sets will be approximated by the corresponding solution sets of the inflated Euler scheme.
For $N\in\N\setminus\{0\}$, set $h_N:=T/N$, define a temporal grid $t_{N,n}=nh_N$, $n=0,\ldots,N$, and let 
$\beta_N,\delta_N\in[0,\infty)$ be numbers to be specified later.
Then any sequence $(y_{N,n})_n\subset\R^d$ satisfying 
\begin{align}
y_{N,n+1} \in y_{N,n} + h_NF(t_{N,n},y_{N,n}) + B_{\beta_N}(0), 
\label{Euler}
\end{align}
is called a solution of the inflated Euler scheme with step-size $h_N$, 
which may as well be subject to a possibly relaxed initial condition
\begin{align}
\dist(y_{N,0}, X_0)\le\delta_N \label{Euler:IC}
\end{align}
and possibly relaxed state constraints
\begin{align}
\dist(y_{N,n},A(t_{N,n}))\le\delta_N. 
\label{Euler:SC}
\end{align}
For $\nu\in\{0,\ldots,N\}$, we denote the solution sets of the discrete unconstrained and the discrete constrained problems 
with time horizon $\nu$ by
\begin{align*}
S^u_N(\nu,\beta_N,\delta_N) := &\{(y_{N,n})_{n=0}^{\nu}\subset\R^d: \text{\eqref{Euler} holds for $0\le n\le \nu-1$, 
\eqref{Euler:IC} holds}\},\\
S^c_N(\nu,\beta_N,\delta_N) := &\{(y_{N,n})_{n=0}^{\nu}\in S^u_N(\nu,\beta_N,\delta_N): \text{\eqref{Euler:SC} 
holds for $0\le n\le \nu$}\},
\end{align*}
and we identify these sequences with continuous functions $y_N(\cdot):[0,\nu h]\rightarrow\R^d$ via piecewise linear interpolation
of the data $(t_{N,n},y_{N,n})_{n=0}^\nu$,
where the representant of the derivative is chosen piecewise constant on intervals of the form $[t_{N,n},t_{N,n+1})$.
 
\medskip

Throughout this paper, we posit the following assumptions on the ordinary differential inclusion and the constraints.
\begin{hypothesis} \label{weak:hypothesis}
The right-hand side $F$, the constraints $A$ and the initial set $X_0$ satisfy the following conditions:
\begin{itemize}
\item [(1)] The mapping $F:[0,T]\times\R^d\rightrightarrows\R^d$ has convex and compact values 
and is $L$-Lipschitz, i.e.\ there exists some $L>0$ such that
\[\dist(F(t,x),F(t',x')) \le L|t-t'| + L\|x-x'\|\ \text{for all}\ t,t'\in[0,T],\ x,x'\in\R^d.\]
\item [(2)] The state constraints $A:[0,T]\rightrightarrows\R^d$ have closed values and are upper semicontinuous in the sense that
$t=\lim_{n\rightarrow\infty}t_n$, $x_n\in A(t_n)$ and $x=\lim_{n\rightarrow\infty}x_n$ imply $x\in A(t)$.
\item [(3)] The set $X_0\subset\R^d$ is compact, and the data $F$, $A$ and $X_0$ are compatible in the sense 
that there exists $\tau\in(0,T]$ with $S^c(\tau)\neq\emptyset$.
\end{itemize}
\end{hypothesis}

We will also posit assumptions on the blowup size $\beta_N$ of the Euler scheme 
and the relaxation parameter $\delta_N$ for the initial condition and the constraints.
\begin{hypothesis} \label{blowup:hypothesis}
The parameters $\beta_N$ and $\delta_N$ satisfy the following conditions:
\begin{itemize}
\item [(1)] The blowup size satisfies $\tfrac{\beta_N}{h_N}\to 0$ as $N\to\infty$.
\item [(2)] The relaxation parameter satisfies $\delta_N\to 0$ as $N\to\infty$.
\end{itemize}
\end{hypothesis}
The relaxation parameter $\delta_N$ will not be needed before Section \ref{space:disc}.
We introduce it at this early stage, because we wish to provide a-priori bounds
in Lemma \ref{a:priori:bounds}, which apply to all exact and numerical trajectories 
in this paper.

\section{Auxiliary results} \label{auxiliary}

Since $S^c(\tau)\subset S^u(\tau)$ and 
$S^c_N(\nu,\beta_N,\delta_N)\subset S^u_N(\nu,\beta_N,\delta_N)$,
the bounds given below hold for the constrained problems as well.
\begin{lemma} \label{a:priori:bounds}
We have the a-priori bounds 
\[\dist(x(t),X_0) \le \hat R(t):= \tfrac{1}{L}(e^{Lt}-1)(\sup_{x\in X_0}\|F(0,x)\|+1)\quad \forall t\in[0,\tau]\]
for all $x(\cdot)\in S^u(\tau)$ and
\[\dist(y_{N,n},X_0) \le \hat R_N(n)
:= e^{Lt_{N,n}}(\delta_N + \tfrac{\beta_N}{Lh_N} + \tfrac{1}{L}\sup_{x\in X_0}\|F(0,x)\|+1)\quad 
\forall n\in\{0,\ldots,\nu\}\]
for all $(y_{N,n})_{n=0}^\nu\in S_N^u(\nu,\beta_N,\delta_N)$.
Moreover, we have
\[R:=\max\{1,\tfrac{1}{L}\}e^{LT}(\sup_{x\in X_0}\|F(0,x)\|+ \sup_{N\in\N}(\delta_N+\tfrac{\beta_N}{Lh_N})+1)<\infty,\]
we have $\hat R(t)\le R$ for all $t\in[0,T]$ and $\hat R_N(n)\le R$ for all $n\in\{0,\ldots,N\}$ and $N\in\N$, 
and for all $(t,x)\in[0,T]\times B_R(X_0)$ we have
\[\|F(t,x)\|\le P:= L(T+R)+\sup_{x'\in X_0}\|F(0,x')\| < \infty.\]
\end{lemma}
\begin{proof}
The bounds for $S^u(t)$ and $S_N^u(\nu,\beta_N,\delta_N)$ follow by applying Gronwall's Lemma 
and induction, respectively, in a straight-forward manner. Boundedness of $F$ is a consequence of hypothesis (A1).
\end{proof}

The following lemma quantifies the error between exact trajectories and Euler trajectories with perturbed initial
value. It will be used to determine an appropriate size $\beta_N$ of the blowup for the inflated Euler scheme. 

\begin{lemma} \label{blow:up:size}
Let $t\in[0,T]$ and $h>0$ be such that $[t,t+h]\subset[0,T]$, let $z\in\R^d$, 
and let $x(\cdot)\in W^{1,1}([t,t+h],B_R(X_0))$ be a solution of \eqref{ODI}.
Then there exists $v\in F(t,z)$ such that
\begin{align*}
&\|x(t+h)-(z+hv)\| \le (1+Lh)\|x(t)-z\| + L(P+1)h^2.
\end{align*}
\end{lemma}
\begin{proof}
As the metric projection to a compact and convex set is continuous according to 
\cite[Theorem 9.3.4]{Aubin:Frankowska:90}, we may define
\[v:=\tfrac{1}{h}\int_t^{t+h}\Proj(\dot x(s),F(t,z))ds,\]
and $v \in F(t,z)$ holds by \cite[Theorem I.6.13]{Warga:72}.
Because of Lemma \ref{a:priori:bounds}, we have
\[\|x(s)-x(t)\| \le \int_t^{t+h}\|\dot x(\theta)\|d\theta \le Ph\]
for $s\in[t,t+h]$, and we find
\begin{align*}
&\|\int_t^{t+h}\dot x(s)ds-hv\| \le \int_t^{t+h}\|\dot x(s)-\Proj(\dot x(s),F(t,z))\|ds\\
&= \int_t^{t+h}\dist(\dot x(s),F(t,z))ds \le \int_t^{t+h}\dist(F(s,x(s)),F(t,z))ds\\
&\le \int_t^{t+h}L(s-t)+L\|x(s)-x(t)\|+L\|x(t)-z\|dt\\ 
&\le Lh\|x(t)-z\|+L(P+1)h^2,
\end{align*}
so that
\begin{align*}
\|x(t+h)-(z+hv)\| &\le \|x(t)-z\| + \|\int_t^{t+h}\dot x(s)ds-hv\|\\
&\le (1+Lh)\|x(t)-z\| + L(P+1)h^2.
\end{align*}
\end{proof}

The following result, which is proved by compactness type arguments, is the core of this paper. 
In the current situation, it replaces Filippov's theorem (see \cite[Theorem 2.4.1]{Aubin:Cellina:84}
for a continuous and \cite[Theorem 2.2]{Baier:Chahma:Lempio:07} for a discrete version of this result
for Lipschitz right-hand side), which yields estimates for the semi-distances between numerical end exact 
solutions in the absence of state constraints.
Here and later on, we equip the space $\R\times\R^d\times\R^d$ with the norm $|\cdot|+\|\cdot\|+\|\cdot\|$.

\begin{lemma} \label{Filippov:substitute}
Let $(\tau_N)_N\subset(0,T]$ be a sequence with $\tau:=\lim_{N\rightarrow\infty}\tau_N>0$,
let $c(\cdot)\in L^1(0,T)$ and let $x_N(\cdot)\in W^{1,1}([0,\tau_N],\R^d)$ be a sequence of functions satisfying
\begin{align} 
&\|\dot x_N(t)\|\le c(t)\ \text{for almost every}\ t\in(0,\tau_N),\label{a1}\\
&\sup_{t\in[0,\tau_N]}\dist((t,x_N(t),\dot x_N(t)),\graph(F))\to 0\ \text{as}\ N\to\infty,\label{a2}\\
&\dist(x_N(0),X_0)\rightarrow 0\ \text{as}\ N\to\infty,\label{a3}\\
&\sup_{n\in\{0,\ldots,\lfloor\tfrac{\tau_N}{h_N}\rfloor\}}\dist(x_N(t_{N,n}), A(t_{N,n}))\to 0\ \text{as}\ N\to\infty. \label{a4}
\end{align}
Then there exists some $x(\cdot)\in S^c(\tau)$ and functions $\tilde x_N(\cdot)\in W^{1,1}([0,\tau],\R^d)$ 
that coincide with $x_N(\cdot)$ on $[0,\tau_N]$ such that
\begin{align*}
&\sup_{t\in[0,\tau]}\|\tilde x_N(t)-x(t)\|\to 0\ \text{as}\ N\to\infty,\\
&\dot{\tilde{x}}_N(\cdot)\rightharpoonup \dot x(\cdot)\ \text{in}\ L^1((0,\tau),\R^d)
\end{align*}
along a subsequence.
\end{lemma}

\begin{proof}
Define 
\[\tilde x_N(t):=\left\{\begin{array}{ll}x_N(t),& t\in[0,\min\{\tau_N,\tau\}],\\ 
x_N(\tau_N),& t\in(\min\{\tau_N,\tau\},\tau].\end{array}\right.\]
Then the sequence $(\tilde x_N(\cdot))_N\subset AC([0,\tau],\R^d)$ satisfies condition \eqref{a1} with $\tau$ instead of $\tau_N$,
and it inherits conditions \eqref{a2}, \eqref{a3} and \eqref{a4} with $\min\{\tau_N,\tau\}$ instead of $\tau_N$.
Because of \eqref{a1}, the functions $(\tilde x_N(\cdot))_N$ are uniformly bounded, and \cite[Theorem 0.3.4]{Aubin:Cellina:84}
guarantees that there exists some $x(\cdot)\in W^{1,1}([0,\tau],\R^d)$ such that 
\begin{align}
&\sup_{t\in[0,\tau]}\|\tilde x_N(t)-x(t)\|\to 0, \label{local:a}\\
&\tilde x_N(\cdot)\ \rightharpoonup x(\cdot)\ \text{in}\ L^1((0,\tau),\R^d) \nonumber
\end{align}
along a subsequence.
By \eqref{a2}, the convergence theorem \cite[Theorem 1.4.1]{Aubin:Cellina:84} applies, and together with
\eqref{a3} and compactness of $X_0$, we obtain $x(\cdot)\in S^u(\tau)$.

Let us check that $x(\cdot)\in S^c(\tau)$.
For any $t\in[0,\tau)$, we have $t_{N,\lfloor t/h_N\rfloor}\le t\le\min\{\tau_N,\tau\}$ for almost all $N$ and 
$t=\lim_{N\to\infty}t_{N,\lfloor t/h_N\rfloor}$.
By \eqref{a4}, the points $z_N:=\Proj(\tilde x_N(t_{N,\lfloor t/h_N\rfloor}),A(t_{N,\lfloor t/h_N\rfloor})$ satisfy
\begin{align} \label{local:b}
&\|\tilde x_N(t_{N,\lfloor t/h_N\rfloor})-z_N\| \nonumber \\
&= \dist(\tilde x_N(t_{N,\lfloor t/h_N\rfloor}), A(t_{N,\lfloor t/h_N\rfloor}))\to 0\ \text{as}\ N\to\infty.
\end{align}
By continuity of $x(\cdot)$ and because of statements \eqref{local:a} and \eqref{local:b}, we have
\begin{align*}
&\|x(t)-z_N\|\\ 
&\le \|x(t)-x(t_{N,\lfloor t/h_N\rfloor})\| + \|x(t_{N,\lfloor t/h_N\rfloor})-\tilde x_N(t_{N,\lfloor t/h_N\rfloor})\|
+ \|\tilde x_N(t_{N,\lfloor t/h_N\rfloor})-z_N\|\\
&\to 0\ \text{as}\ N\rightarrow\infty,
\end{align*}
so that hypothesis (A2) implies 
\[x(t)=\lim_{N\to\infty}z_N\in A(t).\]
Since $t\in[0,\tau)$ was arbitrary, we also have $x(\tau)\in A(\tau)$ by continuity of $x(\cdot)$ and hypothesis (A2).
Consequently, we have $x(\cdot)\in S^c(\tau)$.
\end{proof}

As a consequence, we obtain the following statement about the life span of the solutions to the the state constrained 
differential inclusion.
\begin{corollary}\label{maximum}
The number 
\[\tau^*:=\sup\{\tau\in(0,T]: S^c(\tau)\neq\emptyset\}\]
is, in fact, a maximum.
\end{corollary}
\begin{proof}
By hypothesis (A3), there exists $\tau\in(0,T]$ with $S^c(\tau)\neq\emptyset$, 
so that $\tau^*$ is well-defined.
By definition of the supremum, there exist $x_N(\cdot)\in S^c(\tau_N)$, $N\in\N$, 
such that $\tau_N\to\tau^*$. 
According to Lemma \ref{Filippov:substitute}, they possess a limit, which is a member of $S^c(\tau^*)$.
\end{proof}

\section{The inflated Euler scheme} \label{inflated}

The following example shows that it is necessary to modify the Euler scheme
to treat state constrained problems.
\begin{example} \label{oscillator}
Consider the constrained differential equation
\begin{align*}
&\dot x_1 = -x_2,\quad \dot x_2 = x_1,\quad x_1^2+x_2^2 =1,\quad x_1\le 0,\quad x_1(0)=0,\quad x_2(0)=1,
\end{align*}
with unique solution 
\[x_1(t)=\sin(t+\pi/2),\quad x_2(t)=\cos(t+\pi/2),\]
and maximal interval of existence $[0,\pi]$.
For any step-size $h>0$, the Euler scheme maps $(0,1)$ to the point
$(-h,1)$, which does not satisfy the constraints, so that no feasible 
approximate trajectory exists.
\end{example}

Motivated by Lemma \ref{blow:up:size} and the above example, we choose a blowup size
\begin{equation} \label{local:c}
\beta_N:=L(P+1)h_N^2.
\end{equation}
In this context, there is no need to relax the initial condition or the state constraints,
so that we have $\delta_N=0$.
A relaxation will become necessary in Section \ref{space:disc}, where not only the time interval,
but also the phase space is discretized.

The following proposition states that with this choice of $\beta_N$, 
the exact dynamics are overapproximated by the inflated Euler scheme.
In particular, Euler trajectories exist at least as long as exact trajectories.

\begin{proposition}\label{overapproximate}
For any $x(\cdot)\in S^c(\tau)$ with $\tau\in(0,\tau^*]$, there exists a sequence $(y_N)_N$ with 
$y_N\in S_N^c(\lfloor\tau/h_N\rfloor,\beta_N,0)$ such that 
\begin{align}
\label{local:3}
x(t_{N,n})=y_{N,n},\quad  n=0,\ldots,\lfloor \tau/h_N\rfloor. 
\end{align}
\end{proposition}

\begin{proof}
The existence of a sequence $(y_N)_N$ with \eqref{local:3} can be obtained by successively applying 
Lemma \ref{blow:up:size} and taking \eqref{local:c} into account.
\end{proof}

The statement of Proposition \ref{overapproximate} can be reformulated in terms 
of the Hausdorff semidistance as
\begin{equation*}
\sup_{x(\cdot)\in S^c(\tau)}
\inf_{y_N\in S^c_N(\lfloor\tau/h_N\rfloor,\beta_N,0)}
\max_{n=0,\ldots,\lfloor\tau/h_N\rfloor}\|y_{N,n}-x(t_{N,n})\| = 0.
\end{equation*}
The following proposition guarantees that the opposite Hausdorff semidistance between numerical and exact
trajectories converges to zero.
In particular, Euler trajectories do not live longer asymptotically than exact trajectories.

\begin{proposition} \label{error:prop}
For any $\tau\in(0,\tau^*]$, we have
\begin{equation}\label{conv:err}
\sup_{y_N\in S^c_N(\lfloor\tau/h_N\rfloor,\beta_N,0)}
\inf_{x(\cdot)\in S^c(\tau)}
\max_{n=0,\ldots,\lfloor\tau/h_N\rfloor}\|y_{N,n}-x(t_{N,n})\|\to 0
\end{equation}
as $N\to\infty$.
For any $\tau\in(\tau^*,T]$, there exists $N_\tau\in\N$ such that 
\[S^c_N(\lfloor\tau/h_N\rfloor,\beta_N,0)=\emptyset\quad\forall N\ge N_\tau.\]
\end{proposition}

\begin{proof}
Let $\tau\in(0,\tau^*]$. 
By Corollary \ref{maximum} and because of Proposition \ref{overapproximate}, we have
\begin{align*}
S^c(\tau)\neq\emptyset\quad\text{and}\quad
S^c_N(\lfloor\tau/h_N\rfloor,\beta_N,0)\neq\emptyset.
\end{align*}
Assume that statement \eqref{conv:err} is false.
Then there exist $\epsilon>0$ and a sequence 
$y_N\in S^c_N(\lfloor\tau/h_N\rfloor,\beta_N,0)$, 
$N\in\N'\subset\N$, such that 
\begin{equation}\label{contradiction}
\inf_{x(\cdot)\in S^c(t_{N,\lfloor\tau/h_N\rfloor})} \max_{n=0,\ldots,\lfloor\tau/h_N\rfloor}\|y_{N,n}-x(t_{N,n})\|>\epsilon.
\end{equation}
The linearly interpolated Euler trajectories $y_N(\cdot)$ satisfy the assumptions of Lemma \ref{Filippov:substitute}, 
because we have $t_{N,\lfloor\tau/h_N\rfloor}\to\tau$, the estimate
\[\|\dot y_N(s)\| \le P\quad\forall s\in[0,t_{N,\lfloor\tau/h_N\rfloor}]\]
holds according to Lemma \ref{a:priori:bounds}, we have
\begin{eqnarray*}
&& \dist((s,y_N(s),\dot y_N(s)),\graph(F)) \\
&\le& \|(s,y_N(s),\dot y_N(s))-(t_{N,\lfloor s/h_N\rfloor},y_N(t_{N,\lfloor s/h_N\rfloor}),\dot y_N(t_{N,\lfloor s/h_N\rfloor}))\| \\
&&+\dist((t_{N,\lfloor s/h_N\rfloor},y_N(t_{N,\lfloor s/h_N\rfloor}),\dot y_N(t_{N,\lfloor s/h_N\rfloor}),\graph(F))\\
&\le& (1+P)h_N+\beta_N/h_N\to 0\ \text{as}\ N\to\infty
\end{eqnarray*}
by Hypothesis (\ref{blowup:hypothesis}1), and the initial condition as well as the state constraints are satisfied by definition.
Hence, by Lemma \ref{Filippov:substitute}, there exist $x(\cdot)\in S^c(\tau)$, a subsequence $\N''\subset\N'$ 
and functions $\tilde y_N(\cdot)\in W^{1,1}([0,\tau],\R^d)$,
$N\in\N''$, that coincide with $y_N(\cdot)$ on $[0,t_{N,\lfloor\tau/h_N\rfloor}]$ such that
\begin{align*}
&\sup_{t\in[0,\tau]}\|\tilde y_N(t)-x(t)\|\to 0\ \text{as}\ \N''\ni N\to\infty.
\end{align*}
This contradicts assumption \eqref{contradiction}, and hence statement \eqref{conv:err} is correct.

The second assertion is proved in a very similar way.
Assume that $\tau>\tau^*$ and that there exist $y_N(\cdot)\in S^c_N(\lfloor\tau/h_N\rfloor,\beta_N,0)$, $N\in\N'$,
for a subsequence $\N'\subset\N$.
As above, we can extract another subsequence converging to a solution $x(\cdot)\in S^c(\tau)$,
which contradicts $S^c(\tau)=\emptyset$.
\end{proof}

\section{Spatial discretization} \label{space:disc}

When the solution sets are computed in practice, it is necessary to discretize the underlying state space.
Simple pathological examples show that in this situation, it is necessary to relax the state constraints in order to 
obtain convergence of the numerical scheme.
The results obtained in this context as well as their proofs are similar to those in Section \ref{inflated}.

In addition to Hypotheses \ref{weak:hypothesis} and \ref{blowup:hypothesis}, we posit the following assumptions 
on the discretization of the state space.
\begin{hypothesis} \label{grid:hypothesis}
The spatial nets $\Delta_N\subset\R^d$ satisfy the following conditions:
\begin{itemize}
\item [(1)] For every $x\in\R^d$, there exists some $y^\#\in\Delta_N$ with $\|x-y^\#\|\le\delta_N$.
\item [(2)] The relaxation parameter $\delta_N$ satisfies $\delta_N/h_N\to 0$ as $N\to\infty$.
\end{itemize}
\end{hypothesis}
Consider the discrete solution sets
\[S^{c,\#}_N(\nu,\beta_N^\#,\delta_N):= \{y_N\in S^c_N(\nu,\beta_N^\#,\delta_N): 
y_{N,n}\in\Delta_N,\ n=0,\ldots,\nu\}\]
with $0\le\nu\le N$ and
\[\beta_N^{\#}:=\beta_N+(2+Lh_N)\delta_N=L(P+1)h_N^2+(2+Lh_N)\delta_N.\]
The following computations show that $\beta_N^\#>0$
is a good choice for the blowup parameter.
Hypothesis (\ref{grid:hypothesis}2) strengthens Hypothesis (\ref{blowup:hypothesis}2) and ensures that $\beta_N^\#$ 
satisfies Hypothesis (\ref{blowup:hypothesis}1).

The next proposition implies that the fully discrete Euler trajectories
are as close to the exact ones as the grid size allows. 
In particular, they do not cease to exist before the exact solutions violate the constraints.

\begin{proposition}\label{rho:2:estimate}
For any $x(\cdot)\in S^c(\tau,X_0)$ with $\tau\in(0,\tau^*]$, there exists a sequence $(y_N^\#)_N$ with 
$y_N^\#(\cdot)\in S^{c,\#}_N(\lfloor\tau/h_N\rfloor,\beta_N^\#,\delta_N)$ such that 
\begin{align}
\|x(t_{N,n})-y_N^\#(t_{N,n})\|\le\delta_N,\quad  n=0,\ldots,\lfloor \tau/h_N\rfloor. \label{local:d}
\end{align}
\end{proposition}

\begin{proof}
We construct the desired trajectory recursively.
By Hypothesis (\ref{grid:hypothesis}1), there exists an element 
\[y^\#_{N,0}\in\Delta_N\cap B_{\delta_N}(x(0))\subset\Delta_N\cap B_{\delta_N}(X_0)\cap B_{\delta_N}(A(0))\]
which satisfies the relaxed initial condition \eqref{Euler:IC} as well as the relaxed state constraints \eqref{Euler:SC}
and statement \eqref{local:d}.
Assume that $y^\#_{N,0},\ldots,y^\#_{N,n}$ with $(y^\#_{N,k})_{k=0}^n\in S^{c,\#}_N(n,\beta_N^\#,\delta_N)$ and
\[\|x(t_{N_k})-y^\#_{N,k}\|\le\delta,\quad k=0,\ldots n,\]
have been constructed and that $n<\nu$. 
Then Lemma \ref{blow:up:size} yields
\begin{equation} \label{local:4}
\dist(x(t_{N_,n+1}),y^\#_{N,n}+h_NF(t_{N,n},y^\#_{N,n}))\le(1+Lh_N)\delta_N+L(P+1)h_N^2. 
\end{equation}
By Hypothesis (\ref{grid:hypothesis}1), there exists a point 
\[y^\#_{N,n+1}\in\Delta_N\cap B_{\delta_N}(x(t_{N,n+1}))\subset\Delta_N\cap B_{\delta_N}(A(t_{N,n+1}))\]
satisfying the relaxed state constraint \eqref{Euler:SC}.
Because of \eqref{local:4}, we have
\[y^\#_{N,n+1} \in y^\#_{N,n}+h_NF(t_{N,n},y^\#_{N,n})+B_{\beta^\#}(0),\]
so that $y^\#_{N,0},\ldots,y^\#_{N,n+1}$ is a trajectory of the fully discrete Euler scheme.
\end{proof}

The following proposition shows that the opposite Hausdorff semi-distance between fully discrete and exact
trajectories converges to zero.
In particular, fully discrete Euler trajectories do not exist longer asymptotically than exact trajectories.

\begin{proposition}
For any $\tau\in(0,\tau^*]$, we have
\begin{equation*}
\sup_{y_N^\#\in S^{c,\#}_N(\lfloor\tau/h_N\rfloor,\beta_N^\#,\delta_N)}
\inf_{x(\cdot)\in S^c(\tau)}
\max_{n=0,\ldots,\lfloor\tau/h_N\rfloor}\|y_N(t_{N,n})-x(t_{N,n})\|\to 0
\end{equation*}
as $N\to\infty$.
For any $\tau>\tau^*$, there exists $N_\tau\in\N$ such that
\[S^{c,\#}_N(\lfloor\tau/h_N\rfloor,\beta_N^\#,\delta_N)=\emptyset\quad\forall N\ge N_\tau.\]
\end{proposition}

\begin{proof}
The proof is the same as that of Proposition \ref{error:prop}, provided the assumptions of 
Lemma \ref{Filippov:substitute} can be verified in the present setting. 
But for any sequence $(y_N^\#)_N$ with $y_N^\#\in S^{c,\#}_N(\lfloor\tau/h_N\rfloor,\beta_N^\#,\delta_N)$,
the linearly interpolated trajectories $y_N^\#(\cdot)$ satisfy
\begin{align*}
&\sup_{N\in\N}\sup_{t\in(0,t_{N,\lfloor\tau/h_N\rfloor})}\|\dot y_N^\#(t)\| \le P+\sup_{N\in\N}\beta_N^\#/h_N<\infty
\end{align*}
as well as
\begin{eqnarray*}
&&\dist((t,y_N^\#(t),\dot y_N^\#(t)),\graph(F))\\
&\le& \|(t,y_N^\#(t),\dot y_N^\#(t))-(t_{N,\lfloor t/h_N\rfloor},y_N^\#(t_{N,\lfloor t/h_N\rfloor}),\dot y_N^\#(t_{N,\lfloor t/h_N\rfloor}))\|\\
&&+\dist((t_{N,\lfloor t/h_N\rfloor},y_N^\#(t_{N,\lfloor t/h_N\rfloor}),\dot y_N^\#(t_{N,\lfloor t/h_N\rfloor})),\graph(F))\\
&\le& (L+1)(P+1)h_N+2\delta_N/h_N+L\delta_N\to 0\quad\text{as}\ N\to\infty,
\end{eqnarray*}
and by the relaxed initial condition and state constraints, we have
\begin{align*}
&\dist(y_N^\#(0),X_0) \le \delta_N\to 0\quad\text{as}\ N\to\infty,\\
&\sup_{n=0,\ldots,N}\dist(y_N^\#(t_{N,n}),A(t_{N,n}))\le\delta_N\to 0\quad\text{as}\ N\to\infty,
\end{align*}
so that Lemma \ref{Filippov:substitute} is applicable.
\end{proof}

\begin{figure}[p]
\begin{center}
\includegraphics[scale=0.75]{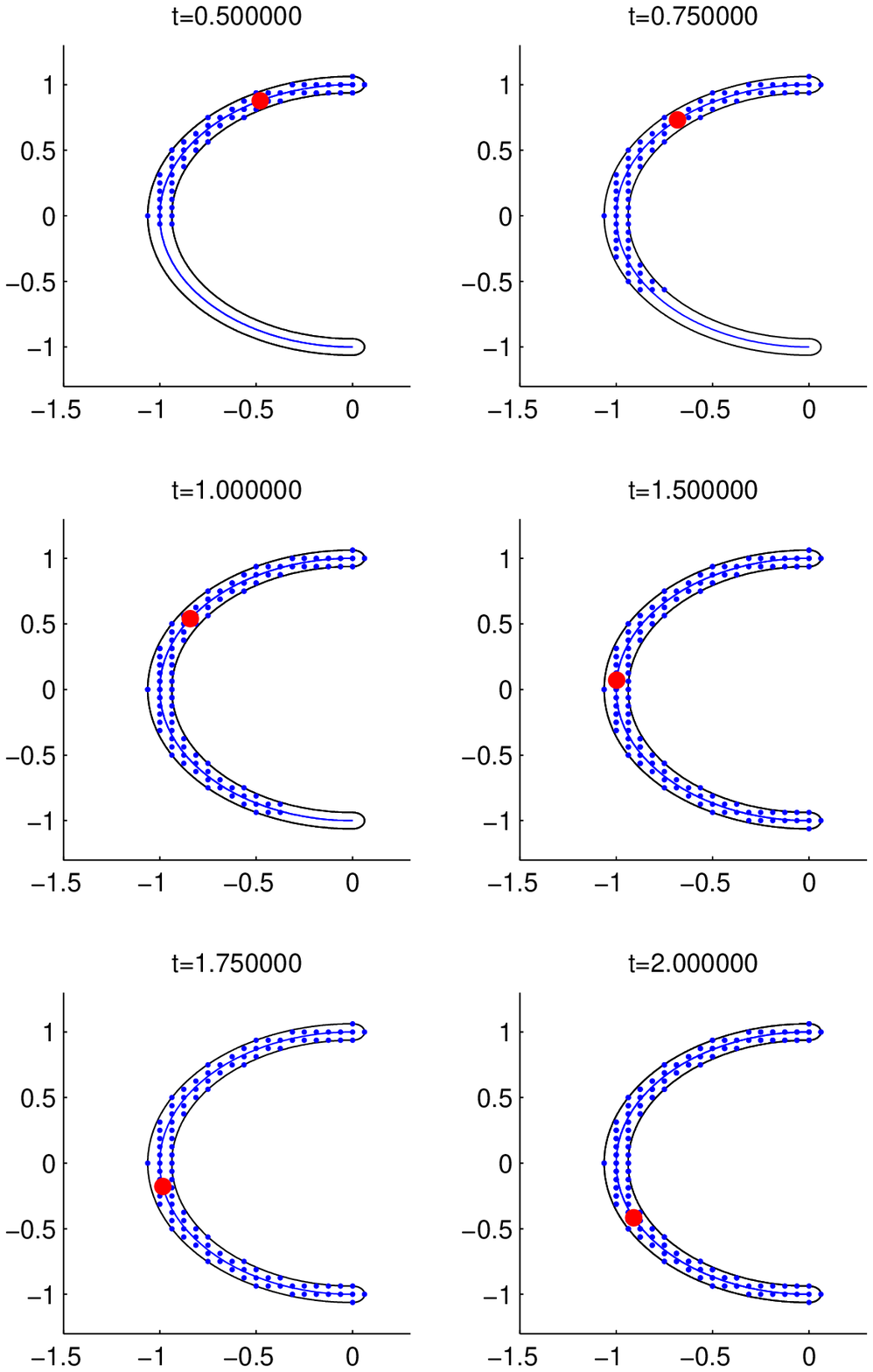}
\end{center}
\caption{Approximate reachable sets for Example \ref{oscillator} with step-size $h=0.25$.\label{025}}
\end{figure}

\begin{figure}[p]
\begin{center}
\includegraphics[scale=0.75]{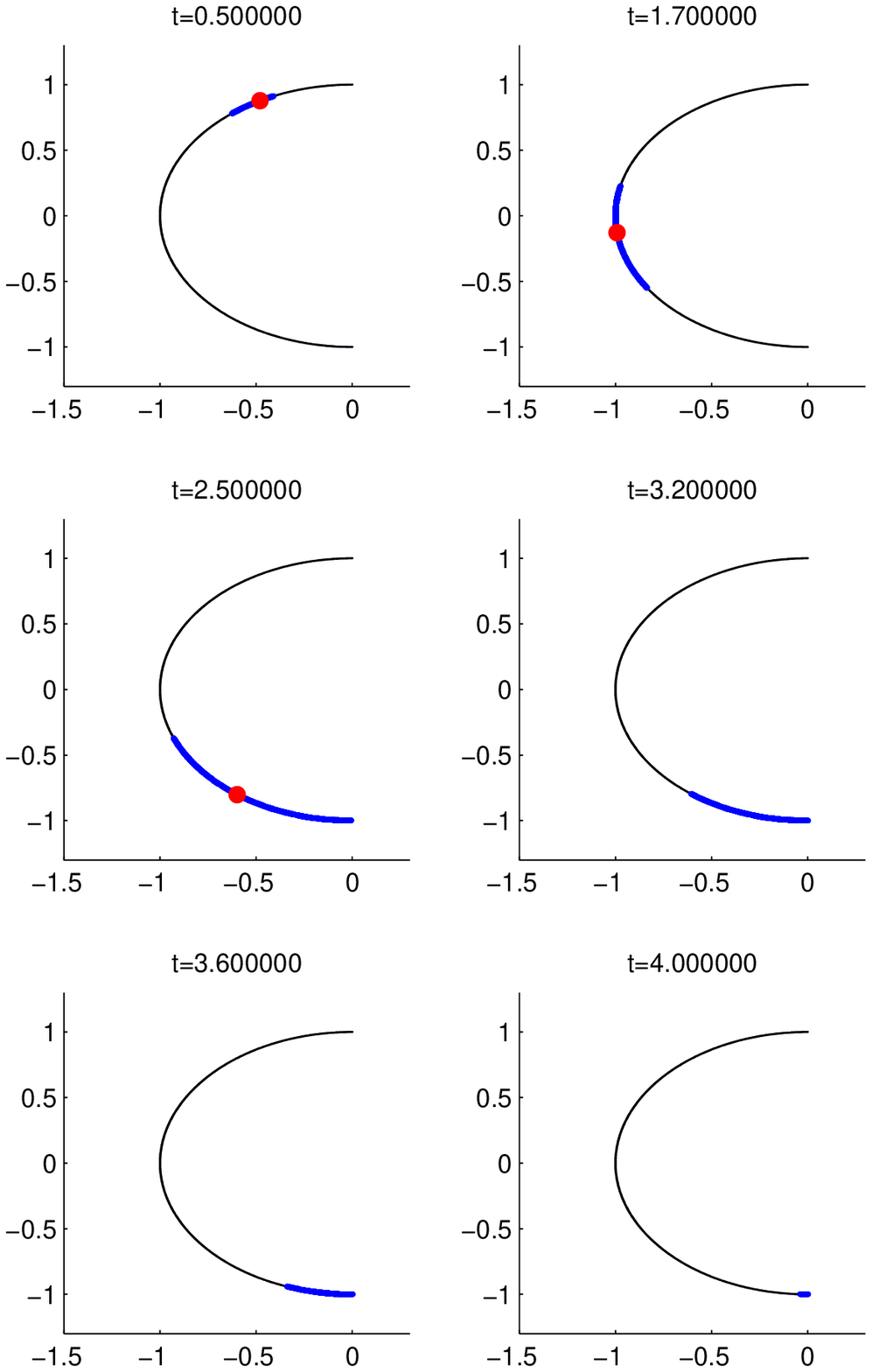}
\end{center}
\caption{Approximate reachable sets for Example \ref{oscillator} with step-size $h=0.05$.\label{005}}
\end{figure}

\begin{figure}[p]
\begin{center}
\includegraphics[scale=0.75]{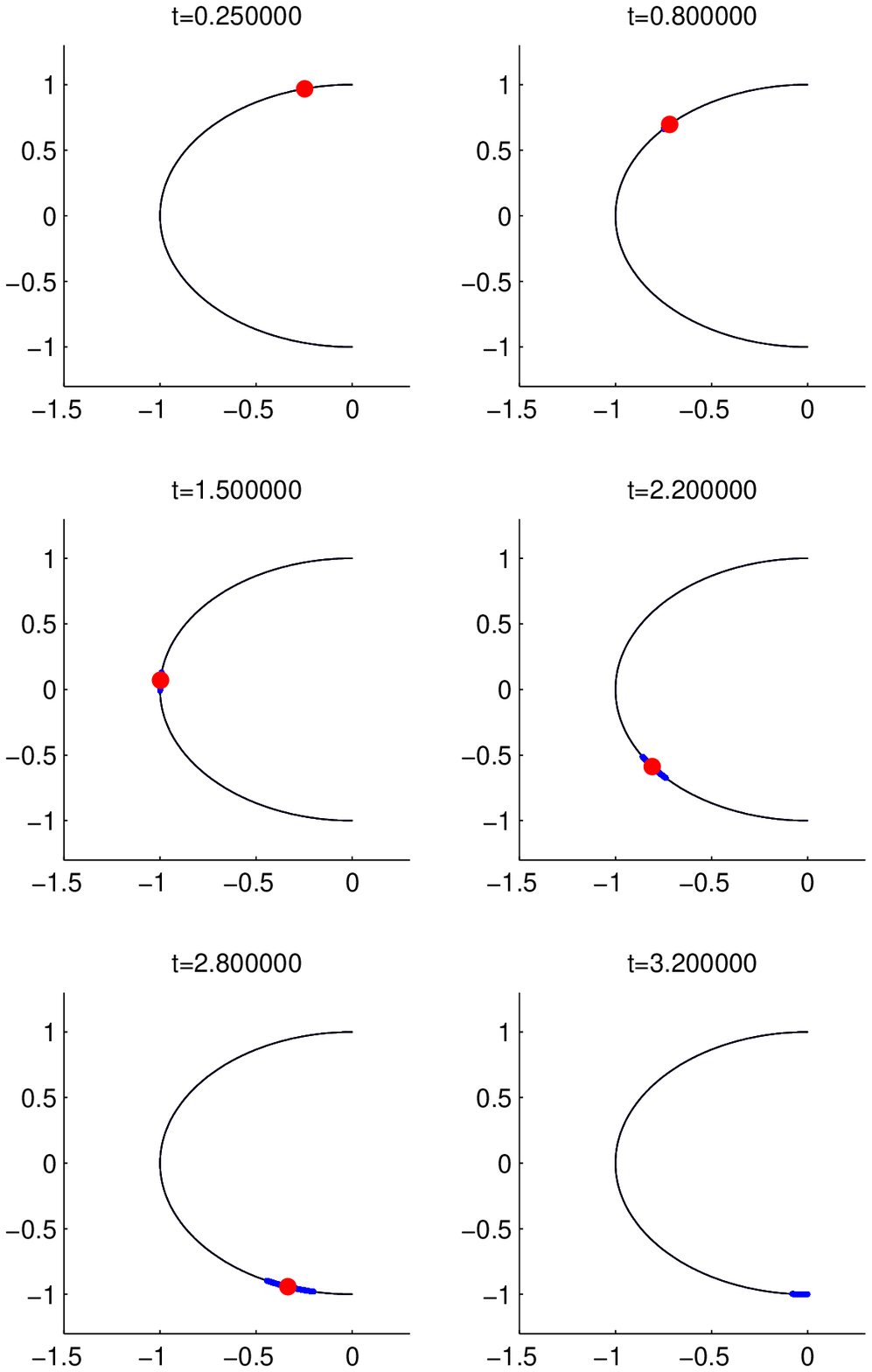}
\end{center}
\caption{Approximate reachable sets for Example \ref{oscillator} with step-size $h=0.01$.\label{001}}
\end{figure}

\begin{figure}[p]
\begin{center}
\begin{footnotesize}
\begin{tabular}[b]{cccc} 
    \toprule
    $h_N$ & $\beta_N$ & $\delta_N$ & $\tau_N^*$ \\ 
    \midrule
    0.2500 & 0.3281 & 0.0625 & $\infty$ \\
    0.2000 & 0.2080 & 0.0400 & $\infty$ \\
    0.1800 & 0.1678 & 0.0324 & 20.1600 \\
    0.1500 & 0.1159 & 0.0225 & 10.3500 \\
    0.1200 & 0.0737 & 0.0144 & 7.0800 \\
    0.1000 & 0.0510 & 0.0100 & 6.0000 \\
    0.0750 & 0.0285 & 0.0065 & 4.8000 \\
    0.0500 & 0.0126 & 0.0025 & 4.1000 \\
    0.0250 & 0.0031 & 0.0006 & 3.5500 \\
    0.0100 & 0.0005 & 0.0001 & 3.2900 \\
    \bottomrule
\end{tabular}
\end{footnotesize}
\end{center}
\caption{Data for Example \ref{oscillator} with constants $L=1$ and $P=2$.
Here, $\tau_N^*$ denotes the maximal interval of existence for the numerical trajectories.\label{data}}
\end{figure}

\begin{figure}[p]
\begin{center}
\includegraphics[scale=0.7]{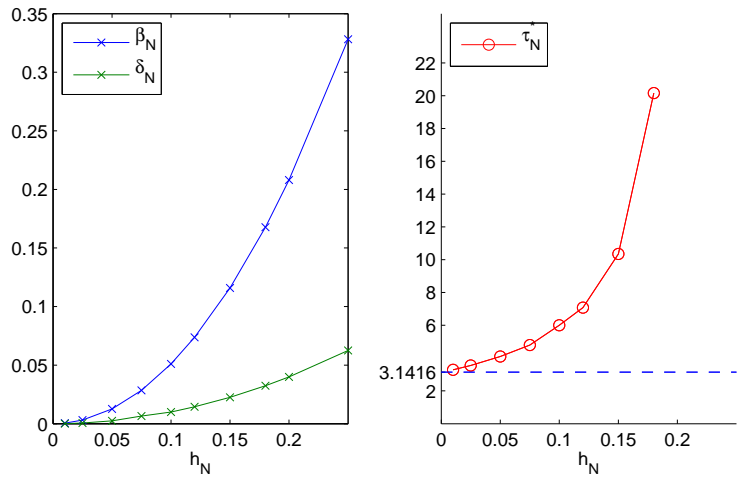}
\end{center}
\caption{Visualized data for Example \ref{oscillator}.\label{dataplot}}
\end{figure}

\section{Numerical example} \label{num:ex}

We continue the simple, but instructive Example \ref{oscillator} from Section \ref{inflated}.
As it is difficult to visualize sets of trajectories, Figures \ref{025}, \ref{005} and \ref{001} show the reachable sets 
\[R^{c,\#}_N(\nu,\beta_N^\#,\delta_N)=\{y_{N,\nu}: (y_{N,n})_{n=0}^\nu\in S^{c,\#}_N(\nu,\beta_N^\#,\delta_N)\}\]
of the spatially discretized inflated Euler scheme for step-sizes $h=0.25$, $h=0.05$ and $h=0.01$, i.e.\ the sets of all 
points that can be reached by trajectories of the numerical scheme at time $t=t_{N,\nu}$.

In the figures, the blue semiarc depicts the admissible set 
\[A=\{x\in\R^2:x_1^2+x_2^2 =1,\ x_1\le 0\},\]
which is constant in time.
The black line drawn around the semiarc represents the boundary of the admissible set $A+B_{\delta_N}(0)$
of the numerical scheme.
Blue dots are elements of the discrete reachable sets $R^{c,\#}_N(\nu,\beta_N^\#,\delta_N)$, and
the centers of the large red dots indicate the position of the exact solution at time $t=t_{N,\nu}$.
Whenever dots exceed their corresponding admissible sets, this is due to limited plotting precision.

The step-sizes have been chosen in such a way that we can observe the typical behavior of the inflated Euler scheme,
as predicted by the analysis in the sections above.
The exact solution is always approximated by a numerical trajectory up to the mesh-size of the
spatial grid, which implies that some approximate trajectories live at least as long as the exact solutions.
The approximate reachable sets, however, can be substantially larger than that of the exact dynamics.
Moreover, numerical trajectories may live much longer than their exact counterparts.
Figure \ref{025} shows an extreme case in which the reachable set of the numerical scheme becomes stationary
after time $t=1.5$, and hence some trajectories are defined on $\R_+$, while the exact solution ceases to
exist at time $\tau^*=\pi$.
In Figures \ref{005} and \ref{001}, we see that the discrete reachable sets converge towards the true 
solution as $h_N\to 0$ and that the life span
\[\tau_N^*:=\max\{t_{N,\nu}: S^{c,\#}_N(\nu,\beta_N^\#,\delta_N)\neq\emptyset\}\]
of the numerical scheme converges to the exact value $\tau^*=\pi$ as $h_N\to 0$.
This behavior is clearly recognizable in the data given in Figures \ref{data} and \ref{dataplot},
where the inflation and relaxation parameters $\beta_N^\#$ and $\delta_N$ as well as 
the life span $\tau_N^*$ are plotted as functions of $h_N$.

\section{Exploiting stability theorems} \label{stability:sec}

Many papers in optimal control theory are concerned with the relationship between the solution sets $S^u(\tau)$ and $S^c(\tau)$
for $\tau\in[0,T]$. 
The specification of conditions, under which property (\ref{Bressan}1) below holds, is subject of ongoing research.
An overview over the relevant literature as well as instructive counterexamples can be found in \cite{Bettiol:Bressan:Vinter:10}.

It is certain that conditions, which ensure property (\ref{Bressan}1), are not satisfied by the class of examples we are aiming at,
because property (\ref{Bressan}1) implies $S^c(\tau)\neq\emptyset$ for arbitrarily large $\tau$.
Nevertheless, in this section, we will assume this property as given and describe briefly its consequences for the inflated Euler scheme.

The ordinary Euler scheme (without inflation) has been analyzed in this context in \cite{Baier:Chahma:Lempio:07},
so that similar results should also be made available for the inflated scheme.
We would also like to provide a template, how property (\ref{Bressan}1) can be exploited to obtain short and simple
proofs for linear convergence of a numerical scheme.

To this end, we suppose the following hypothesis in addition to Hypotheses \ref{weak:hypothesis} and \ref{blowup:hypothesis}.
\begin{hypothesis} \label{Bressan}
The solution sets $S^u(\tau)$ and $S^c(\tau)$ enjoy a linear stability property, and the constraints are Lipschitz:
\begin{itemize}
\item [(1)] For any $\tau\in(0,\tau^*]$, there exists some $K=K_\tau>0$ such that for any $x^u(\cdot)\in S^u(\tau)$, there exists
some $x^c(\cdot)\in S^c(\tau)$ with
\[\max_{t\in[0,\tau]}\|x^u(t)-x^c(t)\| \le K\max_{t\in[0,\tau]}\dist(x^u(t),A(t)).\]
\item [(2)] The state constraints are $L_A$-Lipschitz, i.e.\ there exists $L_A>0$ with
\[\dist(A(t_1),A(t_2))\le L_A|t_1-t_2|\quad\forall t_1,t_2\in[0,T].\]
\end{itemize}
\end{hypothesis}

As Proposition \ref{overapproximate} yields that the first Hausdorff semidistance between the exact and the numerical 
solution sets is zero, we only need to take care of the second semidistance.
The proof of Proposition \ref{linear:prop} is remarkably short compared to the machinery 
set up in \cite{Baier:Chahma:Lempio:07} to prove a similar estimate under similar assumptions.

\begin{proposition} \label{linear:prop}
For any $\tau\in(0,\tau^*]$, there exists $C>0$ such that
\begin{equation} \label{linear}
\sup_{y_N\in S^c_N(\lfloor\tau/h_N\rfloor,\beta_N,0)}
\inf_{x(\cdot)\in S^c(\tau)}
\max_{n=0,\ldots,\lfloor\tau/h_N\rfloor}\|y_{N,n}-x(t_{N,n})\| \le Ch_N.
\end{equation}
\end{proposition}
\begin{proof}
A short computation shows that for arbitrary $y_N\in S^c_N(\lfloor\tau/h_N\rfloor,\beta_N,0)$, $N\in\N$,
the residual of the interpolated Euler trajectory $y_N(\cdot)$ inserted into the differential inclusion \eqref{ODI} satisfies
\[\dist(\dot y_N(t),F(t,y_N(t))) \le 2L(P+1)h_N\quad\forall t\in[0,t_{N,\lfloor\tau/h_N\rfloor}].\]
Therefore, Filippov's theorem \cite[Theorem 2.4.1]{Aubin:Cellina:84} guarantees that there exist $C_1=C_1(L,P)>0$ 
independent of the choice of $y_N$ and a solution $x_N^u(\cdot)\in S^u(t_{N,\lfloor\tau/h_N\rfloor})$
such that
\[\max_{t\in[0,t_{N,\lfloor\tau/h_N\rfloor}]}\|y_N(t)-x^u_N(t)\|\le C_1h_N.\]
Since $y_{N,n}\in A(t_{N,n})$ holds for all $n\in\{0,\ldots,\lfloor\tau/h_N\rfloor\}$, 
Lemma \ref{a:priori:bounds} and Hypothesis (\ref{Bressan}2) imply 
\begin{eqnarray*}
\dist(y_N(t),A(t)) 
&\le& \|y_N(t)-y_N(t_{N,\lfloor t/h_N\rfloor})\|  
+\dist(y_N(t_{N,\lfloor t/h_N\rfloor}),A(t_{N,\lfloor t/h_N\rfloor}))\\
&&+\dist(A(t_{N,\lfloor t/h_N\rfloor}),A(t)) \le (P+L_A)h_N,
\end{eqnarray*}
so that 
\[\dist(x_N^u(t),A(t))\le C_2h_N\quad\forall t\in[0,t_{N,\lfloor\tau/h_N\rfloor}]\]
with $C_2=C_1+P+L_A$.
By Hypothesis (\ref{Bressan}1), there exist $x_N^c(\cdot)\in S^c(t_{N,\lfloor\tau/h_N\rfloor})$ and a constant
$K>0$ independent of the particular $x_N^u(\cdot)$, such that
\[\|x_N^u(t)-x_N^c(t)\|\le KC_2h_N\quad\forall t\in[0,t_{N,\lfloor\tau/h_N\rfloor}],\]
and hence estimate \eqref{linear} follows with $C=C_1+KC_2$.
\end{proof}

It is not difficult to see that the above result persists under a spatial discretization with parameter $\delta_N=h_N^2$.
In this case, the first Hausdorff semidistance is still covered by Proposition \ref{rho:2:estimate}.

\section{Conclusion}

The aim of the present paper was to prove convergence of the inflated Euler scheme in the presence
of state constraints and under weak assumptions, not excluding applications such as the collision avoidance 
problems discussed in \cite{Gerdts:Xausa:13} and 
\cite{Landry:Gerdts:Henrion:Hoemberg:13}.
This aim is achieved, but the convergence statements are not as strong as one might wish,
because we only showed pure convergence. 
It would be desirable to know if there were any conditions on the right-hand side and the state constraints
that are substantially weaker than those imposed in \cite{Baier:Chahma:Lempio:07}, 
\cite{Bettiol:Bressan:Vinter:10} and \cite{Bettiol:Frankowska:Vinter:12}, but strong enough 
to specify a rate or a speed of convergence of some Euler-like scheme.
In that sense, we leave a gap in the literature, which we are not able to fill at the moment.

\bibliographystyle{plain}
\bibliography{constraints}

\begin{thebibliography}{10}

\bibitem{Aubin:Cellina:84}
J.-P. Aubin and A.~Cellina.
\newblock {\em Differential inclusions}, volume 264 of {\em Grundlehren der
  Mathematischen Wissenschaften}.
\newblock Springer-Verlag, Berlin, 1984.

\bibitem{Aubin:Frankowska:90}
J.-P. Aubin and H.~Frankowska.
\newblock {\em Set-valued analysis}, volume~2 of {\em Systems \& Control:
  Foundations \& Applications}.
\newblock Birk\-h\"au\-ser Boston, Inc., Boston, MA, 1990.

\bibitem{Baier:Chahma:Lempio:07}
R.~Baier, I.A. Chahma, and F.~Lempio.
\newblock Stability and convergence of {E}uler's method for state-constrained
  differential inclusions.
\newblock {\em SIAM J. Optim.}, 18(3):1004--1026, 2007.

\bibitem{Baier:Gerdts:Xausa:13}
R.~Baier, M.~Gerdts, and I.~Xausa.
\newblock Approximation of reachable sets using optimal control algorithms.
\newblock {\em Numer.\ Algebra Control Optim.}, 3(3):519--548, 2013.

\bibitem{Bettiol:Bressan:Vinter:10}
P.~Bettiol, A.~Bressan, and R.~Vinter.
\newblock On trajectories satisfying a state constraint: {$W^{1,1}$} estimates
  and counterexamples.
\newblock {\em SIAM J. Control Optim.}, 48(7):4664--4679, 2010.

\bibitem{Bettiol:Frankowska:Vinter:12}
P.~Bettiol, H.~Frankowska, and R.~Vinter.
\newblock {$L^\infty$} estimates on trajectories confined to a closed subset.
\newblock {\em J. Differential Equations}, 252(2):1912--1933, 2012.

\bibitem{Beyn:Rieger:07}
W.-J. Beyn and J.~Rieger.
\newblock Numerical fixed grid methods for differential inclusions.
\newblock {\em Computing}, 81(1):91--106, 2007.

\bibitem{Beyn:Rieger:10}
W.-J. Beyn and J.~Rieger.
\newblock The implicit {E}uler scheme for one-sided {L}ipschitz differential
  inclusions.
\newblock {\em Discrete Contin. Dyn. Syst. Ser. B}, 14(2):409--428, 2010.

\bibitem{Donchev:Farkhi:98}
T.~Donchev and E.~Farkhi.
\newblock Stability and {E}uler approximation of one-sided {L}ipschitz
  differential inclusions.
\newblock {\em SIAM J.\ Control Optim.}, 36(2):780--796, 1998.

\bibitem{Dontchev:Lempio:92}
A.~Dontchev and F.~Lempio.
\newblock Difference methods for differential inclusions: a survey.
\newblock {\em SIAM Rev.}, 34(2):263--294, 1992.

\bibitem{Gerdts:Xausa:13}
M.~Gerdts and I.~Xausa.
\newblock Avoidance trajectories using reachable sets and parametric
  sensitivity analysis.
\newblock In Dietmar H\"omberg and Fredi Tr\"oltzsch, editors, {\em System
  Modeling and Optimization}, volume 391 of {\em IFIP Advances in Information
  and Communication Technology}, pages 491--500. Springer Berlin Heidelberg,
  2013.

\bibitem{Grammel:03}
G.~Grammel.
\newblock Towards fully discretized differential inclusions.
\newblock {\em Set-Valued Anal.}, 11(1):1--8, 2003.

\bibitem{Landry:Gerdts:Henrion:Hoemberg:13}
C.~Landry, M.~Gerdts, R.~Henrion, and D.~H\"omberg.
\newblock Path-planning with collision avoidance in automotive industry.
\newblock In Dietmar H\"omberg and Fredi Tr\"oltzsch, editors, {\em System
  Modeling and Optimization}, volume 391 of {\em IFIP Advances in Information
  and Communication Technology}, pages 102--111. Springer Berlin Heidelberg,
  2013.

\bibitem{Lempio:Veliov:98}
F.~Lempio and V.~Veliov.
\newblock Discrete approximations of differential inclusions.
\newblock {\em Bayreuth. Math. Schr.}, (54):149--232, 1998.

\bibitem{Rieger:14a}
J.~Rieger.
\newblock Semi-implicit {E}uler schemes for ordinary differential inclusions.
\newblock {\em SIAM J.\ Numer.\ Anal.}, 52(2):895--914, 2014.

\bibitem{Rieger:14}
J.~Rieger.
\newblock Robust boundary tracking for reachable sets of nonlinear differential
  inclusions.
\newblock {\em Found. Comput. Math.}, to appear.

\bibitem{Sandberg:08}
M.~Sandberg.
\newblock Convergence of the forward {E}uler method for nonconvex differential
  inclusions.
\newblock {\em SIAM J. Numer. Anal.}, 47(1):308--320, 2008.

\bibitem{Szolnoki:03}
D.~Szolnoki.
\newblock Set oriented methods for computing reachable sets and control sets.
\newblock {\em Discrete Contin. Dyn. Syst. Ser. B}, 3(3):361--382, 2003.

\bibitem{Warga:72}
J.~Warga.
\newblock {\em Optimal Control of Differential and Functional Equations}.
\newblock Academic Press, New York, 1972.

\end{thebibliography}

\end{document}